\documentclass[12pt]{amsart} 

\usepackage[dvips]{graphics} 
\usepackage{color,amssymb,amsbsy,amsmath,amsfonts,amssymb,
amscd,epsfig,times,graphics,enumerate}

\newtheorem{theo}{Theorem}
\newtheorem{prop}{Proposition}[section]
\newtheorem{lem}{Lemma}[section]
\newtheorem{rem}{Remark}[section]
\theoremstyle{definition}
\newtheorem{defi}{Definition}[section]
\newtheorem{example}{Example}
 
\numberwithin{equation}{section} 

 \newcommand{\Om}{\Omega}

\newcommand{\C}{\mathbb C} 
\newcommand{\R}{\mathbb R}

\newcommand{\be}{\begin{equation}}
\newcommand{\ov}{\overline}

\newcommand{\ee}{\end{equation}}

\newcommand{\beq}{\begin{eqnarray*}}
\newcommand{\eeq}{\end{eqnarray*}}

\begin{document}
  
\begin{abstract}
We generalize Lempert's and Poletsky's works \cite{lem, pol} on the description of extremal discs for the Kobayashi metric
to a higher order setting with $k$-stationarity condition introduced in \cite{be-de}.

\end{abstract} 

\title[Extremal and stationary discs for the Kobayashi $k$-metric]
{Extremal and stationary discs for the Kobayashi $k$-metric} 

\author{Florian Bertrand, Giuseppe Della Sala, Jae-Cheon Joo}

\subjclass[2010]{32F45, 32Q45}
\keywords{}
\maketitle 
\section*{Introduction}
The Kobayashi pseudodistance was introduced by S. Kobayashi \cite{ko1, ko2} for the purpose of finding a 
generalization of the Poincar\'e distance of the unit disc in more general complex spaces. It turned out later that the 
Kobayashi pseudodistance is an inner pseudodistance and H. L. Royden \cite{ro} defined an infinitesimal pseudometric, called the 
Kobayashi pseudometric, whose integrated pseudodistance coincides with the Kobayashi pseudodistance. Due to 
their invariance by biholomorphisms, the Kobayashi pseudodistance and metric are particularly adapted  to study properties of 
holomorphic mappings and structures of complex spaces (cf. \cite{ko3,ja-pf}). 
In his celebrated paper, L. Lempert 
\cite{lem} described extremal discs for the Kobayashi metric in smooth bounded strongly convex domains of $\C^n$  
as stationary discs and used it to define a canonical representation of such domains onto the unit 
ball.    
Lempert's work has given rise to  studies of the relation between extremality and stationarity conditions in more general situations (see for instance \cite{pol, hu, pa, tu, co-ga-su}). In particular, E. Poletsky \cite{pol} proved that the stationarity condition is the Euler-Lagrange equation corresponding to the Kobayashi extremal problem on bounded domains of $\C^n$ defined by $C^2$ plurisubharmonic functions.

In \cite{be-de}, the first two authors introduced a higher order notion of stationarity condition in order to study finite jet 
determination of CR automorphisms of Levi degenerate hypersurfaces. Independently, the study on
 higher order Kobayashi metrics has been developed by several authors \cite{ven, yu1}. 
In view of the works of Lempert and Poletsky, it seems natural to study the relation between generalized stationary discs and extremal discs for higher order Kobayashi metrics. In this paper, we prove that the stationarity condition introduced in \cite{be-de} is the Euler-Lagrange equation corresponding to the higher order Kobayashi extremal problem on bounded domains of $\C^n$ defined by $C^2$ plurisubharmonic functions (Theorem  \ref{theoextr}). 
Moreover, we show that on smooth strictly convex domains of $\C^n$, generalized stationary discs are 
locally extremal for the higher order Kobayashi metric (Theorem  \ref{theokstat}). 
Note that the present paper is focused on the higher order Kobayashi metric introduced by S. Venturini \cite{ven}. A study of the metric introduced by  J. Yu \cite{yu1}, by means of extremal discs, was carried out by M. Jarnicki and P. Pflug in \cite{ja-pf}.

\vskip 0,1cm
\noindent  {\it Acknowledgments.}  Research of the first two authors is partly supported by a fellowship 
at the  Center for Advanced Mathematical Sciences (CAMS) at the American University of Beirut. Parts of the paper were 
written while the third author visited the Center for Advanced Mathematical Sciences; he thanks this institution for its support and 
hospitality.

\section{Preliminaries}

For $r>0$, we set $\Delta_r=\{\zeta\in \C \ | \ |\zeta| <r \}$ and we denote by $\Delta=\{\zeta\in \C \ | \  |\zeta| <1 \}$ the  unit disc in $\C$.

\subsection{Jet bundles}

Let $M$ be a complex manifold of dimension $n$. Locally, we identify $M$ with an open domain in $\C^n$ by taking local coordinates $(z^1,...,z^n)$. Let $U_1$ and $U_2$ be two open neighborhoods of the origin in $\C$. For two holomorphic mappings $f:U_1\rightarrow M$ and $g: U_2 \rightarrow M$ and for a positive integer $k$, we denote $f\sim_k g$ if 
$$f^{(\ell)}(0) = g^{(\ell)}(0)$$ 
for all $\ell=0,...,k$. Obvioulsy, $\sim_k$ is an equivalence relation, and the equivalence class of $f$ will be denoted by $J^k_p(f)$, where $p=f(0)\in M$ and is called a {\em $k$-jet at $p$}. The space of all $k$-jets at $p$ is denoted by $J_p^k(M)$ and we define by 
$J^k(M)=\bigcup_{p\in M} J_p^k(M)$ the {\em $k$-jet bundle over $M$}.  
Then the local coordinate system on $M$ enables us to identify $J_p^k(M)$ with $(\C^n)^k$ by the mapping $J^k_p(f)\mapsto (f'(0),\ldots,f^{(k)}(0))$. 

Although the $k$-jet bundle is not a vector bundle over $M$ unless $k=1$, we can still define a complex multiplication as follows. For $\lambda \in \C$, we set 
$$\lambda \cdot J_p^k(f):= J^k_p(f_\lambda)$$
where $f_\lambda(\zeta) = f(\lambda\zeta)$. Locally, if $J^k_p(f) = \xi=(\xi_1, \xi_2,\ldots,\xi_k)\in (\C^n)^k$ then 
$$\lambda \cdot J^k_p(f) = (\lambda\xi_1, \lambda^2 \xi_2,\ldots, \lambda^k \xi_k).$$

\subsection{Higher order Kobayashi metrics} Let $M$ be a complex manifold. Following S. Venturini \cite{ven}, we define for a positive integer $k$,
the {\it Kobayashi $k$-pseudometric}  $K_M^k(p, \xi)$ by 
$$K^k_M\left(p,\xi\right):=\inf 
\left\{\frac{1}{\lambda}>0 \ | \  f: \Delta \rightarrow M 
\mbox{ holomorphic}, f\left(0\right)=p, J^k_p(f)=\lambda\cdot\xi\right\}$$
for any $p\in M$ and any $\xi\in J^k_p(M)$.
Note that $K^1_M$ is the classical Kobayashi pseudometric defined on the tangent bundle. From the definition, it is clear that  $K^k_M(p, c v) = |c| K^k_M(p, v)$ for any $c\in \C$. The following properties are standard and their proofs are straightforward.
\begin{prop} 
\begin{itemize}
\item[(\romannumeral1)] Let $M$ and $N$ be two complex manifolds and $\Phi:M\rightarrow N$ a holomorphic mapping. Then for any $p\in M$ and $\xi\in J_p^k(M)$, we have 
$$K_N^k(\Phi(p), \Phi_* (\xi)) \leq K_M^k (p, \xi),$$ 
where $\Phi_*(\xi) = J^k_{\Phi(p)}(\Phi\circ g)$, if $\xi=J^k_p(g)$.

\item[(\romannumeral2)]  Let $M$ be a complex manifold and $p\in M$.
Let $\tilde \xi \in J^{k+1}_p(M)$ and let $\xi=\pi_k(\tilde \xi)\in J^k_p(M)$, where $\pi_k$ is the canonical projection from $J^{k+1}_p(M)$ onto $J^{k}_p(M)$.   
 Then  we have 
$$ K_M^k(p,\xi) \leq K_M^{k+1}(p,\tilde \xi).$$ 
\end{itemize}
\end{prop}


As an example, we consider in the case of the unit disc $M=\Delta$ and $k=2$. We first need the following lemma
\begin{lem}[The second order Schwarz lemma]\label{l;schwarz}
Let $f: \Delta\rightarrow \Delta$ be a holomorphic function that $f(0)=0$. Then we have $|f'(0)|\leq 1$ and 
\be\label{e;extremal}
|f''(0)|\leq 2(1-|f'(0)|^2).
\ee
Moreover, if the equality holds in \eqref{e;extremal}, then 
\begin{equation}\label{eqextr}
f(\zeta) = \zeta \varphi (e^{i\theta} \zeta)
\end{equation} for some $\theta\in \R$, where $\displaystyle \varphi(\zeta) = \frac{\zeta + f'(0)}{1+ \overline{f'(0)} \zeta}$.
\end{lem}

\begin{proof} The classical Schwarz lemma implies that $|f'(0)|\leq 1$ and if $|f'(0)|=1$, then  $f(\zeta) = e^{i\theta} \zeta$ for some $\theta\in\R$. There is nothing more to prove in this case.  Now we assume $|f'(0)|<1$. 
Let $g(\zeta) = f(\zeta)/\zeta$. Then the Schwarz lemma implies again that $g$ is a holomorphic function from 
$\Delta$ into $\Delta$. Let $a=f'(0)=g(0)$ and let 
$$\psi(\zeta) = \frac{\zeta-a}{1-\bar a \zeta}, \quad h(\zeta) = \psi\circ g(\zeta).$$ 
Since $h(0)=0$, we have 
$$|h'(0)| = \frac{|f''(0)|}{2(1-|f'(0)|^2)}\leq 1,$$which is equivalent to \eqref{e;extremal}. If the equality holds in the above inequality, then $h(\zeta)=e^{i\theta}\zeta$ for some $\theta\in\R$. This yields that $f\zeta) = \zeta \varphi(e^{i\theta}\zeta)$ with
 $\varphi=\psi^{-1}$.
\end{proof}

Applying Lemma \ref{l;schwarz}, we now compute $K_\Delta^2(0,\xi)$  for $\xi=(\xi_1, \xi_2)\in J^2_0 (\Delta)$. Let 
$f:\Delta\rightarrow \Delta$ be a holomorphic function such that $f(0)=0$ and 
$$f'(0) = \lambda \xi_1,\quad f''(0) = \lambda^2 \xi_2$$ 
for some $\lambda \in \C$. By Lemma \ref{l;schwarz} we have
$$|\lambda \xi_1|\leq 1,\quad |\lambda^2 \xi_2| \leq 2-2|\lambda \xi_1|^2.$$ Therefore, 
$$|\lambda| \leq \sqrt{2/(|\xi_2| + 2|\xi_1|^2)}$$
 and there exists a disc $f$ of the form (\ref{eqextr}) which attains the equality.  
 This shows that for $\xi=(\xi_1, \xi_2)\in J^2_0 (\Delta)$, we have 
$$K^2_\Delta(0,\xi)) = \sqrt{|\xi_1|^2 + \frac{|\xi_2|}{2}}.$$

\begin{rem} The following higher order Kobayashi metric was introduced by J. Yu in \cite{yu1} (see also \cite{yu2, ki-hw-ki-le, ni, ja-pf})
$$\chi^k_\Om\left(p,v\right):=\inf 
\left\{\frac{1}{\lambda}>0 \ | \  f: \Delta \rightarrow \Om
\mbox{ holomorphic}, f=p+\zeta^k\Psi, \Psi(0)=\lambda v\right\}$$
for $p$ in a domain $\Om \subset \C^n$ and $v \in \C^n$. Note that for $\xi\in J^k_p(\Om)$ of the form $\xi=(0,\cdots,0,v)$ we have 
$$K^k_\Om\left(p,\xi\right)=\chi^k_\Om\left(p,v\right).$$
\end{rem}

\subsection{$k$-stationary discs} 
Let $\Omega = \{\rho<0\}\subset \mathbb C^n$ be a smooth domain, where $\rho$ is a defining function. We set 
$\displaystyle \partial \rho = \left(\frac{\partial \rho}{\partial z_1},\ldots,\frac{\partial \rho}{\partial z_n}\right)$. Let $k$ be 
a positive integer. Following \cite{be-de}, we define
\begin{defi}
A map $f:\overline{\Delta}\to \mathbb C^n$, holomorphic on $\Delta$ and continuous up to $\overline \Delta$, is a \emph{$k$-stationary disc attached to $b\Omega$} if $f(b\Delta)\subset b\Omega$ and if there exists a continuous function $c:b\Delta\to \mathbb R^+$ such that the function  $\zeta\mapsto \zeta^k c(\zeta)\partial \rho(f(\zeta))\in \mathbb C^n$ defined on $b\Delta$ extends holomorphically to $\Delta$.
\end{defi}
These discs generalize the notion of stationary discs introduced by L. Lempert \cite{lem} and are particularly well adapted to study Levi degenerate hypersurfaces \cite{be-de, be-la-de}. In fact, in the present context, we will also need
\begin{defi}
A bounded holomorphic map $f:\Delta\to \mathbb C^n$ is a \emph{$k$-stationary disc attached to $b\Omega$ in the $L^{\infty}$ sense}  if $f(b\Delta)
\subset b\Omega$ a.e. and if there exists a  $L^{\infty}$ function $c:b\Delta\to \mathbb R^+$ such that the function $ \zeta\mapsto \zeta^k 
c(\zeta)\partial \rho(f(\zeta))\in \mathbb C^n$ defined on $b\Delta$ extends holomorphically to $\Delta$.
\end{defi}

\begin{example}
Consider the case $\Om =\Delta$. Let 
$$f(\zeta) = \prod_{j=1}^k \frac{\zeta-a_j}{1-\ov{a_j} \zeta}$$
for some $a_1,\ldots, a_k\in \Delta$. Let $\rho(\zeta) = |\zeta|^2-1$ be the standard defining function of the unit disc $\Delta$. Then 
$$\partial \rho (f(\zeta)) = \ov{f(\zeta)}$$ and so the function 
$$b\Delta\ni \zeta \mapsto \zeta^k \partial \rho (f(\zeta))  = \zeta^k \ov{f(\zeta)}$$
has winding index $0$ around the origin. Therefore, there exists a function $g$ defined on $b\Delta$ such that $\displaystyle e^{g(\zeta)} = \zeta^k \ov{f(\zeta)}$. Let $h$ be a continuous real-valued function on $b\Delta$ for that $g +h$ extends to a holomorphic function on $\Delta$ and let $c(\zeta) = e^{h(\zeta)}$. Then 
$$\zeta^k c(\zeta)  \partial \rho (f(\zeta)) = e^{(g+h)(\zeta)},$$
which is holomorphic on $\Delta$. Therefore, the disc $f$ is $k$-stationary.
\end{example}

\section{Extremal discs for the Kobayashi $k$-pseudometric are $k$-stationary}
For $v=(v_1,\ldots,v_n),w=(w_1,\ldots,w_n)\in \mathbb C^n$, we set 
$$\langle v,w\rangle = \sum_{j=1}^n v_j w_j.$$
For a domain $\Om\subset \C^n$,  we denote by ${\rm Hol}(\Delta,\Om)$ the space of holomorphic maps from 
$\Delta$ to $\Om$, by ${\rm H}^{\infty}(\Delta,\C^n)={\rm Hol}(\Delta,\C^n) \cap {\rm L}^{\infty}(\Delta,\C^n)$ and 
by ${\rm H}^{1}(\Delta,\C^n)={\rm Hol}(\Delta,\C^n) \cap {\rm L}^{1}(b\Delta,\C^n)$. Let $k$ be 
a positive integer.
\begin{defi} Let $\Omega \subset \C^n$ be a domain. A map 
$f: \Delta \rightarrow \Omega$ is a {\it extremal disc for the Kobayashi $k$-metric for the pair $(p,\xi) \in \Omega \times J^k_p(\Om)$} if  $f(0)=p$, $J^k_p(f)=\lambda \cdot \xi$ 
with $\lambda>0$ and if
$g: \Delta \rightarrow \Omega$ is  holomorphic and such that $g(0)=p$, $J^k_p(g)=\mu \cdot \xi$ with $\mu>0$, then $\mu\leq \lambda$.

\end{defi}

Note that in case $\Om \subset \C^n$ is a bounded domain,  Montel's theorem ensures the existence of extremal discs for any pair $(p,\xi) \in \Omega \times J^k_p(\Om)$.
\begin{theo}\label{theoextr}
Let $\Om=\{\rho<0\} \subset \C^n$ be a bounded domain defined by a $C^2$ plurisubharmonic function $\rho$. Let $p\in \Omega$, $\xi \in J^k_p(\Om)\setminus \{0\}$. Then any extremal disc $f$ for the Kobayashi $k$-metric is almost proper, that is $f(b\Delta) \subset b\Om$ a.e., and  $k$-stationary in the $L^\infty$ sense.  
\end{theo}
The proof follows the variational approach developed by E. Poletsky in  \cite{pol}. We point out that in order to describe complex geodesics in 
complex ellipsoids,  A. Edigarian  \cite{ed} generalized Poletsky's theory. It follows from M. Jarnicki and P. Pflug \cite{ja-pf}  (see in particular Remark 11.4.4), that using Edigarian-Poletsky theory, extremal discs for the Yu's $k^{\rm th}$ order Kobayashi metric are $k$-stationary in the  $L^{\infty}$ 
sense.

\subsection{Poletsky extremal problem}
Let $\Om \subset \C^n$ be a bounded domain defined by a $C^2$ plurisubharmonic function. Consider a real-valued functional $\Phi$ defined on ${\rm H}^{\infty}(\Delta,\C^n)$ satisfying 

(A)  $\Phi$ is differentiable,

(B) for any $f \in {\rm H}^{\infty}(\Delta,\C^n)$, the differential $d_f \Phi$ of $\Phi$ at $f$ is of the form 
$d_f \Phi(h)={\rm Re} \int_{b\Delta} \langle h, \omega \rangle d\theta$ where $\omega$  is holomorphic on $\C\setminus \overline{\Delta_r}$ for some $r<1$.  

Consider $N+1$ functionals $\Phi_0,\ldots,\Phi_N$ satisfying conditions (A) and (B) and real numbers $a_1, \ldots, a_N$. The Poletsky extremal problem (P) consists in finding $f_0 \in{\rm Hol}(\Delta,\Om)$ which maximizes $\Phi_{0}$ under the constraints $\Phi_j(f)=a_j$, $1\leq j \leq N$, and $f \in{\rm Hol}(\Delta,\Om)$.

We also recall that, according to the definition given in \cite{pol} (see also \cite{ed}), real-valued linear functionals $F_1,\ldots,F_N$ defined on ${\rm H}^{1}(\Delta,\C^n)$ and of the form $F_j(f)={\rm Re}\int_{b\Delta} \langle f,\omega_j\rangle d\theta$ are called \emph{linearly independent over $b\Delta$} if $\sum_{j=1}^N \lambda_j \omega_j=g$ on $b\Delta$ for certain $\lambda_j\in \mathbb R$ and $g\in {\rm H}^{\infty}(\Delta,\C^n)$ with $g(0)=0$ only when $\lambda_j=0$ for all $1\leq j \leq N$ and  $g\equiv 0$.
\begin{rem}
In \cite{pol}, linear independence is actually defined as a stronger property; the condition $g(0)=0$ is omitted. However, an inspection of the proof of Theorem 3 in \cite{pol} shows that the definition above is enough for its application. Note that in \cite{ed,ja-pf}, linear independence is defined as well with the  condition $g(0)=0$. 
\end{rem}

E. Poletsky proved the following theorem 

\begin{theo}[Theorem 3 p. 330 \cite{pol}] 
Suppose $f_0$ is a solution of the extremal problem (P) and suppose that the differentials $d_{f_0}\Phi_1, \ldots,
d_{f_0}\Phi_N$ are linearly independent over $b\Delta$. Then: 
\begin{enumerate}[i.]
\item $f_0$ is almost proper.   
\item There exist real numbers $\lambda_1,\ldots,\lambda_N$, a $L^{\infty}$ function $c:b\Delta\to \mathbb R^+$ and $g\in {\rm H}^{\infty}(\Delta,\C^n)$ with $g(0)=0$ such that 
$$\sum_{j=1}^N \lambda_j\omega_j+ g = c \partial \rho(f_0)$$ 
a.e. on $\ b\Delta$, where $d_{f_0}\Phi_j = {\rm Re}\int \langle \cdot,\omega_j\rangle d\theta$ for some $\omega_j$ which is holomorphic on $\C\setminus \overline{\Delta_r}$ for some $r<1$.  

\end{enumerate}
\end{theo}

\subsection{Proof of Theorem \ref{theoextr}}
Let  $p\in\mathbb C^n$ and $\xi=(\xi_1,\ldots,\xi_k)\in J^k_p(\Om)\setminus\{0\}$.  We write $\xi_j=(\xi_{j1},\dots,\xi_{jn})\in \mathbb C^n$ and for any holomorphic map $f: \Delta \to \mathbb C^n$, we write $f=(f_1,\ldots, f_n)$.  Let $1\leq j_0 \leq k$ be the smallest integer such that $\xi_{j_0}\neq 0$.   

We formulate the extremality condition as Poletsky extremal problem (P), as follows.
For any $j$ such that $ \xi_j\neq 0$ we choose linearly independent vectors 
$\eta_j^1,\ldots,\eta_j^{n-1}$ (with $\eta_j^\ell=(\eta_{j1}^\ell,\ldots,\eta_{jn}^\ell)$) such that $\langle \xi_j, \eta_j^\ell\rangle = 0$ for all 
$1\leq \ell \leq n-1$ .
We define, for $1\leq h \leq n$, $ 1\leq \ell \leq n-1$ , $j_0\leq m \leq k$  and $j_0+1\leq m' \leq k$ such that $\xi_m\neq 0$, $\xi_{m'}\neq 0$, 
continuous real functionals on ${\rm H}^{\infty}(\Delta,\C^n)$:
\[
\left\{
\begin{array}{lll}
\displaystyle \Phi_{0,h}^1(f)=\displaystyle \frac{1}{2\pi}{\rm Re}\int_{b\Delta}f_h(e^{i\theta})d\theta, \ \  \ \   \ \  \ \ \displaystyle \Phi_{0,h}^2(f)=\displaystyle \frac{1}{2\pi}{\rm Re}\int_{b\Delta}-if_h(e^{i\theta})d\theta, \\ 
\\
\displaystyle \Phi_{m,\ell}^1(f) =  \displaystyle \frac{1}{2\pi}{\rm Re}\int_{b\Delta} \frac{\langle f(e^{i\theta}), \eta_{m}^\ell\rangle}{\zeta^m}d\theta,
\ \  \ \   \ \  \ \
\displaystyle \Phi_{m,\ell}^2(f) = \displaystyle \frac{1}{2\pi}{\rm Re}\int_{b\Delta} -i\frac{\langle f(e^{i\theta}), \eta_{m}^\ell\rangle
}{\zeta^m}d\theta,\\
\\  
\displaystyle \Phi_{j_0}^1(f) = \displaystyle \frac{(j_0)!}{2\pi\|\xi_{j_0}\|^2}{\rm Re}\int_{b\Delta} \frac{\langle f(e^{i\theta}), \overline \xi_{j_0}\rangle}{\zeta^{j_0}}d\theta,\\
\\
\Phi_{m'}^1(f) = \displaystyle \frac{m'!}{2\pi\|\xi_{m'}\|^2}{\rm Re}\int_{b\Delta} \frac{\langle f(e^{i\theta}), \overline \xi_{m'}\rangle}{\zeta^{m'}}d\theta - \left(\Phi_{j_0}^1(f)\right)^{\frac{m'}{j_0}},\\
\\ 
\displaystyle  \Phi_m^2(f) = \displaystyle \frac{1}{2\pi}{\rm Re}\int_{b\Delta} -i\frac{\langle f(e^{i\theta}), \overline \xi_{m}\rangle}{\zeta^m}d\theta.
\end{array}
\right.
\]
For all $j_0+1\leq m \leq k$ such that $\xi_m=0$ and all $j_0+1\leq m'\leq k$ such that  $\xi_{m'}=0$ we define instead
\[
\left\{
\begin{array}{lll}
\Phi_{m,\ell}^1(f) =\displaystyle  \frac{1}{2\pi}{\rm Re}\int_{b\Delta} \frac{f_\ell(e^{i\theta})}{\zeta^m}d\theta,
\ \  \ \   \ \  \ \

\Phi_{m,\ell}^2(f) = \displaystyle \frac{1}{2\pi}{\rm Re}\int_{b\Delta} -i\frac{f_\ell(e^{i\theta})}{\zeta^m}d\theta,\\
\\
\Phi_m^2(f) =\displaystyle  \frac{1}{2\pi}{\rm Re}\int_{b\Delta} -i\frac{f_n(e^{i\theta})}{\zeta^m}d\theta,
\ \  \ \   \ \  \ \

\Phi_{m'}^1(f) = \displaystyle \frac{1}{2\pi}{\rm Re}\int_{b\Delta} \frac{f_n(e^{i\theta})}{\zeta^{m'}}d\theta.
\end{array}
\right.
\]
\begin{lem}\label{varprob} A holomorphic disc $f: \Delta \to \Om$ is extremal for the Kobayashi $k$-metric for the pair 
$(p,\xi)=(p,0,\cdots,0,\xi_{j_0},\cdots,\xi_k) \in \Omega \times J^k_p(\Om)$ if and only if it maximizes $\Phi_{j_0}^1$ under the constraints 

\[
\left\{
\begin{array}{lll}
\Phi^1_{0,h}(f) = {\rm Re}\,p_h\\
\\
\Phi^2_{0,h}(f) = {\rm Im}\,p_h\\
\\
\Phi^1_{m,\ell}(f) = \Phi^2_{m,\ell}(f) = \Phi^2_m(f) = \Phi^1_{m'}(f)=0\\
\\
f\in {\rm Hol}(\Delta,\Omega).\\
\end{array}
\right.
\]
%
\end{lem}
\begin{proof} Expanding each component of $f$ in its Fourier series, we obtain that the conditions $\Phi^1_{0,h}(f) = {\rm Re}\,p_h$, $\Phi^2_{0,h}(f) = {\rm Im}\,p_h$ are equivalent to $f(0)=p$. Moreover, the choice of the vectors $\eta_j^\ell$ implies that, for certain $\mu_m\in \mathbb C$, $f^{(m)}(0)=\mu_m \xi_m$ for all $j_0\leq m \leq k$ whenever $\Phi^1_{m,\ell}(f) = \Phi^2_{m,\ell}(f) = 0$ for all $j_0\leq m\leq k$, $1\leq \ell \leq n-1$. The conditions $\Phi^2_m(f)=0$, ensure further that $\mu_m\in \mathbb R$ for all $j_0\leq m\leq k$. Setting $\mu_{j_0}=\lambda^{j_0}$, the equations $\Phi^1_{m'}(f)=0$ amount to $\mu_{m'} = \lambda^{m'}$ whenever $\xi_{m'}\neq 0$, and to $f^{(m')}(0)=0$ if $\xi_{m'}=0$. Since then $\Phi^1_{j_0}(f)= \lambda^{j_0}$, a disc $f$ is extremal if it maximizes $\Phi^1_{j_0}$ under the above constraints  and $f\in {\rm Hol}(\Delta,\Omega)$. 
\end{proof}

\begin{lem}\label{satisfies}
The problem in Lemma \ref{varprob} satisfies the conditions of the Poletsky extremal problem $(P)$.
\end{lem}
\begin{proof}[Proof of Lemma \ref{satisfies}] 
Note  that the functionals under considerations are all linear with the exception of $\Phi^1_{m'}$ for any $j_0+1\leq m'\leq k$ such that $\xi_{m'}\neq 0$. A straightforward computation yields 
\[d_{f^0}\Phi^1_{m'}(f)=\frac{1}{2\pi}{\rm Re}\int_{b\Delta} \left\langle f(e^{i\theta}),  \left(\frac{m'!}{\|\xi_{m'}\|^2 \zeta^{m'}}\overline \xi_{m'} - C_{m'}(f^0)\frac{j_0! }{\|\xi_{j_0}\|^2\zeta^{j_0}}\overline \xi_{j_0}  \right) \right\rangle d\theta \]
where 
\begin{equation*}
C_{m'}(f^0) = \frac{m'}{j_0}\left(\Phi^1_{j_0}(f^0) \right)^{\frac{m'}{j_0}-1}.
\end{equation*}
for any $j_0+1\leq m'\leq k$ such that $\xi_{m'}\neq 0$. 
Now, let $e_1,\ldots,e_n$ be the standard basis of $\mathbb C^n$ and let $f^0 \in{\rm H}^{\infty}(\Delta,\C^n)$. Define for 
 $1\leq h \leq n$, $ 1\leq \ell \leq n-1$, $j_0\leq m \leq k$  and $j_0+1\leq m' \leq k$ 
  such that $\xi_m\neq 0$, $\xi_{m'}\neq 0$, holomorphic maps $\mathbb C\setminus\{0\}\to \mathbb C^n$ as follows:
 \[
\left\{
\begin{array}{lll}
\displaystyle \omega_{0,h}^1=\frac{1}{2\pi}e_h,\ \ \omega_{0,h}^2=-\frac{i}{2\pi}e_h,\\
\\
\displaystyle \omega_{m,\ell}^1= \frac{1}{2\pi\zeta^m} \eta^\ell_m, \ \ \omega_{m,\ell}^2= -\frac{i}{2\pi\zeta^m} \eta^\ell_m\\
\\
\displaystyle \omega^1_{j_0}=\frac{j_0!}{2\pi\|\xi_{j_0}\|^2\zeta^{j_0}}\overline \xi_{j_0}\\
\\
\displaystyle  \omega^1_{m'} = \frac{m'!}{2\pi\|\xi_{m'}\|^2\zeta^{m'}}\overline \xi_{m'} - C_{m'}(f^0)\omega^1_{j_0}\\
\\
\displaystyle \omega^2_m = -\frac{i}{2\pi\zeta^m}\overline \xi_m
\end{array}
\right.
\]
For all $j_0+1\leq m \leq k$ such that $\xi_m=0$ and all $j_0+1\leq m'\leq k$ such that  $\xi_{m'}=0$ we define instead
\[\omega_{m,\ell}^1=\frac{1}{2\pi\zeta^m}e_{\ell},\ \ \omega_{m,\ell}^2=-\frac{i}{2\pi\zeta^m}e_\ell, \ \ \omega_m^2=-\frac{i}{2\pi\zeta^m}e_n, \ \ \omega_{m'}^1=\frac{1}{2\pi\zeta^{m'}}e_n.\]
 With these definitions, we have 
 $$\Phi^1_{0,h}(f)={\rm Re}\int_{b\Delta} \langle f,\omega_{0,h}^1\rangle
  d\theta, \ \ \ \ \ \ \ \Phi^2_{0,h}(f)={\rm Re}\int_{b\Delta} \langle f,\omega_{0,h}^2\rangle d\theta,  $$
  $$\Phi^1_{m,\ell}(f)={\rm Re}\int_{b\Delta} \langle f,\omega_{m,\ell}^1\rangle
  d\theta, \ \ \ \ \ \ \ \Phi^2_{m,\ell}(f)={\rm Re}\int_{b\Delta} \langle f,\omega_{m,\ell}^2\rangle d\theta,  $$
  $$\Phi^1_{j_0}(f)={\rm Re}\int_{b\Delta} \langle f,\omega_{j_0}^1\rangle
  d\theta, \ \ \ \ \ \ \ d_{f^0}\Phi^1_{m'}(f)={\rm Re}\int_{b\Delta} \langle f,\omega_{m'}^1\rangle d\theta,  
  \ \ \ \ \ \ \ \Phi^2_{m}(f)={\rm Re}\int_{b\Delta} \langle f,\omega_{m}^2\rangle d\theta,  
  $$
  from which the conclusion follows immediately. 
\end{proof}
In order to apply Theorem 2, it remains to prove the following lemma
\begin{lem} For any disc $f^0\in {\rm H}^{\infty}(\Delta,\C^n)$, the linear functionals $\Phi^1_{0,h}$, $\Phi^2_{0,h}$, $\Phi^1_{m,\ell}$, $\Phi^2_{m,\ell}$, $\Phi^1_{j_0}$, $d_{f^0} \Phi^1_{m'}$ and $\Phi^2_m$  are linearly independent on $b\Delta$.
\end{lem}
\begin{proof} Assume that $\xi_m\neq 0$ for all $j_0+1\leq m\leq k$ (the case when $\xi_m=0$ for some $m$ is simpler). Using the exact expression of the $\omega$'s obtained in the proof of Lemma \ref{satisfies}, we need to consider the vector valued equation on $b\Delta$
\[\sum_{h=1}^n (\lambda_{0,h}^1-i\lambda_{0,h}^2)e_h  +\sum_{m=1}^{j_0-1}\frac{1}{\zeta^m}\sum_{\ell=1}^{n-1}(\lambda^1_{m,\ell}-i\lambda^2_{m,\ell})\eta_{m}^\ell +\] 
\[\frac{1}{\zeta^{j_0}} \left(\left(\frac{j_0!\lambda^1_{j_0}}{\|\xi_{j_0}\|}-i\lambda^2_{j_0}\right)\overline\xi_{j_0} + \sum_{\ell=1}^{n-1}(\lambda^1_{j_0,\ell}-i\lambda^2_{j_0,\ell})\eta_{j_0}^\ell - \sum_{m'=j_0+1}^k \frac{j_0!C_{m'}(f^0)\lambda^1_{m'}}{\|\xi_{j_0}\|^2}\overline \xi_{j_0} \right)+\]
\[ +\sum_{m'=j_0+1}^k\frac{1}{\zeta^{m'}}\left(\left (\frac{m'!\lambda^1_{m'}}{\|\xi_{m'}\|^2} - i\lambda^2_{m'}\right)\overline \xi_{m'} + \sum_{\ell=1}^{n-1}(\lambda^1_{m',\ell}-i\lambda^2_{m',\ell})\eta_{m'}^\ell\right) = g\]
%
%
for suitable $\lambda^1_{0,h},\lambda^2_{0,h},\lambda^1_{m,\ell},\lambda^2_{m,\ell},\lambda^2_m,\lambda^1_{j_0},\lambda^1_{m'}\in \mathbb R$ and $g\in {\rm H}^{\infty}(\Delta,\C^n)$ such that $g(0)=0$. Since the Fourier expansion of $g$ is of the form $\sum_{j\geq 1} g_j \zeta^j$, we have immediately that $g\equiv 0$, $\lambda^1_{0,h}=\lambda^2_{0,h}=0$ for $1\leq h \leq n$ and $\lambda^1_{m,\ell}=\lambda^2_{m,\ell}=0$ for 
$1\leq m \leq j_0-1$. Now observe that for all $j_0\leq m\leq k$ the vectors $\overline \xi_m, \eta_m^1,\ldots,\eta_m^{n-1}$ are a basis of $\mathbb C^n$ over $\mathbb C$. Indeed, if $\overline \xi_m = \sum_{\ell=1}^{n-1}\mu_\ell \eta_m^\ell$ we would have 
$$\|\xi_m\|^2 = \langle \xi_m, \overline \xi_m \rangle= \sum_{\ell=1}^{n-1}\mu_\ell\langle \xi_m, \eta_m^\ell \rangle=0$$ which is a contradiction. It follows that $\lambda^1_{m'} = \lambda^2_{m'} = \lambda^1_{m',\ell} = \lambda^2_{m',\ell} = 0$ for all 
$j_0+1\leq m' \leq k$, $1\leq \ell \leq n-1$, and finally that $\lambda^1_{j_0} = \lambda^2_{j_0} = \lambda^1_{j_0,\ell} = 
\lambda^2_{j_0,\ell} = 0$.
\end{proof}
We are now in a position to apply Theorem 2. We 
obtain that a disc $f$ which is extremal for the problem in Lemma \ref{varprob}, and thus for the Kobayashi $k$-metric, must satisfy the Euler-Lagrange equations
\[
\rho(f) = 0\ {\rm \ a.e. \ on } \ b\Delta,\]
\[\sum_{h=1}^n (\lambda_{0,h}^1\omega^1_{0,h}+\lambda_{0,h}^2 \omega^2_{0,h}) +  (\lambda^1_{j_0} \omega^1_{j_0}+\lambda^2_{j_0} \omega^2_{j_0}) + \sum_{\ell=1}^{n-1}(\lambda^1_{1,\ell}\omega^1_{1,\ell}+\lambda^2_{1,\ell}\omega^2_{1,\ell}) +\] 
\[ + \sum_{m'=2}^k \left (   \lambda^1_{m'} \omega^1_{m'} + \lambda^2_{m'}\omega^2_{m'} +\sum_{\ell=1}^{n-1} \lambda^1_{m',\ell} \omega^1_{m',\ell} +\lambda^2_{m',\ell} \omega^2_{m',\ell} \right ) + g = c \partial \rho(f)\ \ {\rm a.e. \ on } \ b\Delta\]
for a suitable $L^{\infty}$ function $c:b\Delta\to \mathbb R^+$ and a certain $g\in {\rm H}^{\infty}(\Delta,\C^n)$ such that $g(0)=0$. Since the maps $\omega$ are all meromorphic with a pole (of order at most $k$) only at $0$, it follows that the map $\zeta^k c \partial \rho(f)$ extends holomorphically to $\Delta$, hence $f$ is $k$-stationary in the $L^\infty$-sense. This concludes the proof of Theorem \ref{theoextr}.
\qed \section{$k$-stationary discs are locally extremal for the Kobayashi $k$-pseudometric}

\begin{theo}\label{theokstat}
  Let $\Omega \subset \C^n$ be a smooth strictly convex domain. Then $k$-stationary discs for $\Omega$ are locally extremal for the Kobayashi $k$-metric. 
\end{theo}
\begin{proof}
We follow here the approach developped by L. Lempert (see Proposition $1$ in \cite{lem}). Let $f:\Delta\to \mathbb C^n$ be a $k$-stationary disc. 
Let
$g: \Delta\to \mathbb C^n$ be a holomorphic m ap, close enough to $f$, such that $g\neq f$,  $g(0)=f(0)=p$ and 
  $J^k_p(g) =\lambda\cdot J^k_p(f)$ for some positive $\lambda$. We wish to prove that $\lambda\leq 1$.  Due to the convexity of $\Omega$ we have a.e. on $b\Delta$
$${\rm Re}  \langle f(\zeta)-g(\zeta),\partial \rho(f(\zeta))\rangle >0$$
and thus
$${\rm Re} \langle f(\zeta)-g(\zeta),c(\zeta)\partial \rho(f(\zeta))\rangle={\rm Re} 
\langle \zeta ^{-k}(f(\zeta)-g(\zeta)),\tilde{f}(\zeta) \rangle >0.$$      
This implies that 
$$ {\rm Re} \sum_{j=1}^k \left \langle  \frac{1-\lambda^j}{j!}f^{(j)}(0) \zeta^{j-k},\tilde{f}(\zeta) \right \rangle
+ {\rm Re} \sum_{j=k+1}^{\infty}  \left \langle \frac{1}{j!}(f^{(j)}(0)-g^{(j)}(0))\zeta^{j-k}   ,\tilde{f}(\zeta) \right \rangle >0$$
and so       
$$ {\rm Re} \sum_{j=1}^k \frac{1-\lambda^j}{j!} \left \langle  f^{(j)}(0) e^{i(j-k)\theta} ,\tilde{f}\left(e^{i\theta}\right)\right \rangle
+ {\rm Re} \sum_{j=k+1}^{\infty}  \frac{1}{j!}\left \langle (f^{(j)}(0)-g^{(j)}(0))e^{i(j-k)\theta}   ,\tilde{f}\left(e^{i\theta}\right) \right \rangle >0.$$
Since ${\rm Re} \sum_{j=k+1}^{\infty}  \left \langle \frac{1}{j!}(f^{(j)}(0)-g^{(j)}(0))\zeta^{j-k}   ,\tilde{f}(\zeta) \right \rangle$ 
extends as a harmonic function $u$ on $\Delta$ and since $u(0)=0$, we have  
$$  \sum_{j=1}^k \frac{1-\lambda^j}{j!}  {\rm Re} \int_{0}^{2\pi} \left \langle  f^{(j)}(0),\tilde{f}\left(e^{i\theta}\right)\right \rangle e^{i(j-k)\theta} 
{\rm d}\theta >0$$
which can be written as
$$(1-\lambda) \left[\sum_{j=1}^k \frac{1+\lambda+\cdots \lambda^{j-1}}{j!}  {\rm Re} \int_{0}^{2\pi} \left \langle  f^{(j)}(0),\tilde{f}\left(e^{i\theta}\right)\right \rangle e^{i(j-k)\theta} 
{\rm d}\theta \right]>0.$$
We have
$$(1-\lambda) \left[\sum_{j=1}^k \frac{1+\lambda+\cdots \lambda^{j-1}}{j!}  {\rm Re}\left \langle  f^{(j)}(0), \int_{0}^{2\pi} \tilde{f}\left(e^{i\theta}\right) e^{i(j-k)\theta} {\rm d}\theta \right\rangle 
 \right]>0,$$
 and therefore
 \begin{equation}\label{eqsign}
 (1-\lambda) {\rm Re} \left[\sum_{j=1}^k \frac{1+\lambda+\cdots \lambda^{j-1}}{j!(j-k)!} \left \langle  f^{(j)}(0), \tilde{f}^{(k-j)}(0) \right\rangle 
 \right]>0.
 \end{equation}

\begin{lem}\label{fromHopf}
For $\lambda<1$, let $g_\lambda(\zeta):= f(\lambda\zeta)$. Then we have
\[{\rm Re}  \langle f(\zeta)-g_\lambda(\zeta),\partial \rho(f(\zeta))\rangle = (1-\lambda)k_\lambda(\zeta)\] 
with $k_\lambda(\zeta)\to k_1(\zeta)\neq 0$ for all $\zeta\in b\Delta$.
\end{lem}
\begin{proof}
The function $v=\rho\circ f$ is (strictly) subharmonic and negative on $\Delta$, while it vanishes on $b\Delta$. By the Hopf lemma the radial derivative $\displaystyle \frac{\partial v}{\partial r}(\zeta)$ is non-vanishing for $\zeta\in b\Delta$. Using the chain rule we can write
\[\frac{\partial v}{\partial r}(\zeta) = \partial \rho(f(\zeta))\cdot \frac{\partial f}{\partial r}(\zeta) + \overline \partial \rho(f(\zeta))\cdot \frac{\partial \overline f}{\partial r}(\zeta) =  2{\rm Re}  \left \langle \frac{\partial f}{\partial r}(\zeta),\partial \rho(f(\zeta))\right \rangle. \]
On the other hand
\[\lim_{\lambda\to 1^-}\frac{f(\zeta)-g_\lambda(\zeta)}{1-\lambda}=\lim_{\lambda\to 1^-}\frac{f(\zeta)-f(\lambda\zeta)}{1-\lambda}=|\zeta|\frac{\partial f}{\partial r}(\zeta),\]
hence the statement of the lemma follows by setting $\displaystyle k_1(\zeta)=\frac{1}{2}\frac{\partial v}{\partial r}(\zeta)$ for $\zeta\in b\Delta$.
\end{proof}

Now, we turn  back to the proof of Theorem \ref{theokstat}. Putting $g_\lambda(\zeta)=f(\lambda\zeta)$, with $\lambda<1$, in place of $g(\zeta)$ in the previous computation, and using Lemma \ref{fromHopf}, we obtain
\[0\neq\lim_{\lambda\to 1^-} \underbrace{{\rm Re} \left[\sum_{j=1}^k \frac{1+\lambda+\cdots \lambda^{j-1}}{j!(j-k)!} \left \langle  f^{(j)}(0),  \tilde{f}^{(k-j)}(0) \right\rangle 
 \right]}_{I_\lambda}={\rm Re} \left[\sum_{j=1}^k \frac{j}{j!(j-k)!} \left \langle  f^{(j)}(0),  \tilde{f}^{(k-j)}(0) \right\rangle 
 \right].\]
Due to (\ref{eqsign}) and $\lambda<1$, the terms $I_\lambda$ are positive and therefore the expression above is positive. Let then $g:\Delta\to \mathbb C^n$ be any holomorphic map as above which is close enough to $f$; in particular $J^k_p(g) =\lambda\cdot J^k_p(f)$
 with $\lambda$ close enough to $1$, so that we 
have  
$${\rm Re} \left[\sum_{j=1}^k \frac{1+\lambda+\cdots \lambda^{j-1}}{j!(j-k)!} \left \langle  f^{(j)}(0),  \tilde{f}^{(k-j)}(0) \right\rangle 
 \right]>0,$$
which implies together with  (\ref{eqsign})  that $\lambda<1$.
\end{proof}

\vskip 1cm
{\small
\noindent Florian Bertrand\\
Department of Mathematics, Fellow at the  Center for Advanced Mathematical Sciences (CAMS)\\\
American University of Beirut, Beirut, Lebanon\\{\sl E-mail address}: fb31@aub.edu.lb\\

\noindent Giuseppe Della Sala \\
Department of Mathematics, Fellow at the  Center for Advanced Mathematical Sciences (CAMS)\\\
American University of Beirut, Beirut, Lebanon\\{\sl E-mail address}: 	gd16@aub.edu.lb\\

\noindent Jae-Cheon Joo \\
Department of Mathematics and Statistics\\
King Fahd University of Petroleum and Minerals at Dhahran, Kingdom of Saudi Arabia\\
{\sl E-mail address}: 	jcjoo@kfupm.edu.sa\\

}


\begin{thebibliography}{99}

\bibitem{be-de} F. Bertrand, G. Della Sala, {\it Stationary discs for smooth hypersurfaces of finite type and finite jet determination},  J. Geom. Anal. {\bf 25} (2015),
 2516-2545.
 
 \bibitem{be-la-de} F. Bertrand, B. Lamel, G. Della Sala, {\it Jet determination of smooth CR automorphisms and 
generalized stationary discs},  submitted for publication. 
  
 \bibitem{co-ga-su} B. Coupet, H. Gaussier, A. Sukhov, {\it Riemann maps in almost complex manifolds}, 
Ann. Sc. Norm. Super. Pisa Cl. Sci. (5) {\bf 2} (2003) 761-785. 
  
   \bibitem{ed} A. Edigarian, {\it On extremal mappings in complex ellipsoids}, Ann. Polon. Math. {\bf 62} (1995), 83-96. 
   
   \bibitem{hu}X. Huang, {\it A non-degeneracy property of extremal mappings and iterates of holomorphic 
self-mappings}, Ann. Scuola Norm. Sup. Pisa Cl. Sci. (4) {\bf 21} (1994), 399-419. 
  
   \bibitem{ja-pf} M. Jarnicki, P. Pflug, {\it Invariant distances and metrics in complex analysis}. Second extended edition. De Gruyter Expositions in 
   Mathematics, 9. Walter de Gruyter GmbH \& Co. KG, Berlin, 2013. xviii+861 pp.
  
  
  \bibitem{ki-hw-ki-le} J. J. Kim, I. G. Hwang, J. G. Kim, and J. S. Lee, {\it On the higher order Kobayashi metrics},
   Honam Math. J. {\bf 26} (2004), 549-557.  
 
 
\bibitem{ko1} S. Kobayashi,  {\it Distance, holomorphic mappings and the Schwarz lemma}, 
J. Math. Soc. Japan {\bf 19} (1967), 481-485.

\bibitem{ko2} S. Kobayashi, {\it Invariant distances on complex manifolds and holomorphic mappings}, 
J. Math. Soc. Japan {\bf 19} (1967), 460-480.
     
\bibitem{ko3} S. Kobayashi, {\it Hyperbolic complex spaces.} Grundlehren der Mathematischen Wissenschaften [Fundamental Principles of Mathematical Sciences], {\bf 318}. Springer-Verlag, Berlin, 1998. xiv+471 pp.

\bibitem{lem}L. Lempert, {\it La m\'etrique de Kobayashi et la repr\'esentation des domaines sur la boule} (French. English summary) [The Kobayashi metric and the representation of domains on the ball], Bull. Soc. Math. France {\bf 109} (1981), 427-474.

\bibitem{ni} N. Nikolov, {\it Stability and boundary behavior of the Kobayashi metrics}, Acta Math. Hungar. {\bf 90} (2001), 283-291.

\bibitem{pa}M.-Y. Pang,  {\it Smoothness of the Kobayashi metric of nonconvex domains}, Internat. J. Math. {\bf 4} (1993), 953-987.


\bibitem{pol}E. Poletsky, {\it The Euler-Lagrange equations for extremal holomorphic mappings of the unit disk}, Michigan Math. J. {\bf 30} (1983), 317-333. 

\bibitem{ro} H. L. Royden, {\it Remarks on the Kobayashi metric,} Several complex variables, II (Proc. Internat. Conf., Univ. Maryland, College Park, Md., 1970), pp. 125-137. Lecture Notes in Math., Vol. 185, Springer, Berlin, 1971.


\bibitem{tu} A. Tumanov, {\it Extremal discs and the regularity of CR mappings in higher codimension}, Amer. J. Math. 
{\bf 123} (2001), 445-473.


\bibitem{ven} S. Venturini, {\it The Kobayashi metric on complex spaces}, Math. Ann. {\bf 305} (1996), 25-44.

\bibitem{yu1} J. Yu, {\it Singular Kobayashi metrics and finite type conditions}, Proc. Amer. Math. Soc. {\bf 123} (1995), 121-130.

\bibitem{yu2} J. Yu, {\it Weighted boundary limits of the generalized Kobayashi-Royden metrics on weakly pseudoconvex domains},
 Trans. Amer. Math. Soc. {\bf 347} (1995), 587-614.
 
 
 \end{thebibliography}
\end{document}